\documentclass[10pt]{article}
\usepackage{amsmath,amssymb}
\usepackage{color}
\usepackage{mathrsfs}
\usepackage{mathptmx}
\usepackage{verbatim}
\usepackage{epsfig}
\usepackage{CJK}
\usepackage{bm}
\usepackage{amsfonts}
\usepackage{amsthm}
\input{epsf.sty}
\usepackage{epsfig}
\usepackage{setspace}
\def\<{\langle}
\def\>{\rangle}
\def\p{\partial}

\def\be{\begin{equation}}
\def\ee{\end{equation}}
\def\ba{\begin{array}}
\def\ea{\end{array}}

\newtheorem{theorem}{Theorem}
\newtheorem{lemma}{Lemma}[section]

\newtheorem{dfn}{Definition}[section]

\newtheorem{rem}{Remark}

\numberwithin{equation}{section}

\def\be{\begin{equation}}
\def\ee{\end{equation}}
\def\br{\begin{eqnarray}}
\def\er{\end{eqnarray}}
\def\p{\partial}

\title{KAM theorem for reversible mapping of low smoothness with application \footnote{
E-mail:xlijing@sdu.edu.cn(Li); qjg816@163.com(Qi); xpyuan@fudan.edu.cn(Yuan) }}

\author{ \mbox{Jing \ Li$^{\dag}$ }  \mbox{Jiangang \ Qi$^{\dag}$ }    \mbox{Xiaoping \ Yuan$^\ddag$ \hspace{12pt}} \\
$^\dag$ School of Mathematics and Statistics, Shandong University, Weihai, China\\
$^\ddag$ School of Mathematical Sciences, Fudan University, Shanghai, China}

\begin{document}
 \maketitle
%

%
\begin{abstract} Assume the mapping
$$A:\left\{
      \begin{array}{ll}
        x_{1}=x+\omega+y+f(x,y),\\
        y_{1}=y+g(x,y),
      \end{array}
    \right.
(x, y)\in \mathbb{T}^{d}\times B(r_{0})
$$ is reversible with respect to $G: (x, y)\mapsto (-x, y),$ and $| f | _{C^{\ell}(\mathbb{T}^{d}\times B(r_{0}))}\leq \varepsilon_{0}, | g |_{C^{\ell+d}(\mathbb{T}^{d}\times B(r_{0}))}\leq \varepsilon_{0},$
where $B(r_{0}):=\{|y|\le r_0:\; y\in\mathbb R^d\},$ $\ell=2d+1+\mu$ with $0<\mu\ll 1.$ Then when $\varepsilon_{0}=\varepsilon_{0}(d)>0$ is small enough and $\omega$ is Diophantine, the map $A$ possesses an invariant
 torus with rotational frequency $\omega.$ As an application of the obtained theorem, the Lagrange stability is proved for a class of reversible Duffing equation with finite smooth perturbation.

\end{abstract}

\section{Introduction and Main Results}\label{s1}

Kolmogorov \cite{Kol}, Arnold \cite{Arnold1963} and Moser \cite{Moser1962} established the well-known KAM
theory after their names. Let us begin with a nearly integrable Hamiltonian $H=H_0(y)+\varepsilon\, R(x,y)$ where $x\in\mathbb T^d$ is angle variable and $y$ is
action variable in some compact set of $\mathbb R^d$. Endow $H$ the symplectic structure $d\, y\wedge d\, x$. Assume $H_0$ is non-degenerate in the Kolmogorov's sense: $\text{det}\,\left( \frac{\p^2}{\p y^2}\, H_0(y) \right)\neq 0$. In 1954's ICM, Kolmogorov announced that for any Diophantine vector $\omega:= \frac{\p}{\p y}\, H_0(y)$, the Hamiltonian $H$ possesses an invariant torus which carries quasi-periodic motion with rotational frequency vector $\omega$ provided that $H$ is analytic and $\varepsilon$ is small enough. This result is called Kolmogorov's invariant-tori-theorem.  Kolmogorov  himself gave an outline of proof in \cite{Kol}. Arnold \cite{Arnold1963} gave a detail proof for the Kolmogorov's theorem. Arnold's proof is a little bit different from Kolmogorov's outline. Recently one has found that Kolmogorov's proof is valid and of more merits.  Kolmogorov's basic idea is to overcome the difficulty arising from resonances (small divisors) by Newton iteration method.
The main contribution  of Moser \cite{Moser1962} to the KAM theory was to extend
Kolmogorov's invariant-tori-theorem  to smooth category.  Moser exploited smooth approximation technique
closely related to idea of Nash \cite{Nash1956}  to overcome the loss of regularity due to the inversion of
certain (non-elliptic) differential operators at each Newton iteration step.
In the original work of Moser \cite{Moser1962}, which
deals with twist area-preserving maps (corresponding to the Hamiltonian system
case in ¡°one and a half¡± degrees of freedom), the perturbation was assumed to be
 $C^{333}$. The smoothness assumption (in the twist map case) was later relaxed to five
by R\"ussmann \cite{Russmann1970}. The Moser's theorem with improvement by R\"ussmann is usually called Moser's twist theorem.
For the Hamiltonian case we refer to \cite{Moser1969, Zehnder1976},
and, especially, \cite{Poschel1982},
where Kolmogorov¡¯s theorem is proved under the hypothesis that the perturbation is $C^{\ell}$ with $\ell> 2d$.

It is well-known that a center ( phase space is foliated by $1$-dimensional invariant tori) of planar linear system can changed into a focus by a non-Hamiltonian nonlinear perturbation so that all invariant tori are broken down. From this one sees that the Hamiltonian structure plays an important role in preserving the invariant tori undergoing perturbations. Besides the Hamiltonian structure (or symplectic structure for mappings), there is so-called reversible structure for differential equations or mappings on which KAM theory can be constructed. Moser \cite{Moser1973} and Arnold \cite{Arnold1984} initiated the study of reversible differential equations or reversible mappings. In 1973, Moser \cite{Moser1973} constructed a KAM theorem for
\[
  \dot x=\omega+y+f(x,y),\quad \dot y =g(x,y),\]
 where $f$ and $g$ are analytic in their arguments and reversible with respect to the involution $(x,y)\mapsto (-x,y)$, that is,
\[f(-x,y)=f(x,y),\quad g(-x,y)=-g(x,y).\]
The KAM theory for analytic reversible equations (vector-fields) of more general form was deeply investigated in Sevryuk \cite{Sevryuk1986, Sevryuk95, Sevryuk98, Sevryuk2011, Sevryuk2012} and Broer \cite{Broer2009}. Zhang \cite{Zhang2008} constructed a KAM theorem for a class of reversible equations which are assume to be $C^\ell$ smooth where the low bound of $\ell<\infty$ is not specified. Sevryuk \cite{Sevryuk1986} also studied deeply the KAM theory for reversible mappings. For example,
Sevryuk \cite{Sevryuk1986}
constructed a KAM theorem for a reversible mapping $A$ with respect to $G:$
$$A: \left(
       \begin{array}{c}
         x \\
         y \\
         z \\
       \end{array}
     \right)\mapsto\left(
                     \begin{array}{c}
                       x+\lambda y+f^{1}(x, y, z) \\
                       y+f^{2} (x, y, z)\\
                       z+f^{3}(x, y, z) \\
                     \end{array}
                   \right),\;G: \left(
                                  \begin{array}{c}
                                    x \\
                                    y \\
                                    z \\
                                  \end{array}
                                \right)\mapsto \left(
                                                 \begin{array}{c}
                                                   -x+\alpha^{1}(x, y, z) \\
                                                   y+\alpha^{2}(x, y, z) \\
                                                   z+\alpha^{3}(x, y, z) \\
                                                 \end{array}
                                               \right),
$$
where $(x, y, z)$ is in some domain in $\mathbb{T}^{n}\times \mathbb{R}^{p}\times\mathbb{R}^{q},$ constant $\lambda\in (0, 1]$,
 $f^{j}$ $(j=1, 2, 3)$ and $\alpha^{j} (j=1, 2, 3)$ are real analytic in some domain. Liu \cite{Liu2005} established a KAM theorem for analytic and reversible mapping which is quasi-periodic in $x$. In those works the mappings are required to be analytic.
 Naturally one hopes to construct KAM theory for reversible mapping of finite smoothness. Especially one can ask what the lowest smoothness assumption is for reversible mapping.

 Actually, KAM theorem for reversible mapping of finite smoothness is useful in the study of some ordinary differential equations. Dieckerhoff-Zehnder \cite{Dieckerhoff-Zehnder1987} showed the  Lagrange stability for Duffing equation
 \[\ddot x+x^{2n+1}+\sum_{j=0}^{2n}a_{j}(t)x^{j}=0,\;a_{j}(t)\in C^{\infty}(\mathbb{T}^{1})\]
 using Moser's twist theorem. See  \cite{Laederich-Levi1991, Liu1989, Liu1992, Yuan1995,Yuan1998,Yuan2000} for more details. Levi \cite{Levi1991} generalizes the
polynomial $ x^{2n+1}+\sum_{j=0}^{2n}a_{j}(t)x^{j}$ to any finite smooth function of $g(x,t)$ with some suitable conditions, by using the facts that the mapping is required to be finite smooth rather than analytic in Moser's twist theorem and that the Duffing equation is a Hamiltonian system.
Liu \cite{Liu1991, Liu1998, Liu2005}, Piao \cite{Piao2008} and  Yuan-Yuan \cite{Rong2001} proved the Lagrange stability for the  Duffing equation
\begin{equation}\label{1.1}
\ddot{x}+\left(\sum_{j=0}^{[(n-1)/2]}{b_{j}}(t)x^{2j+1}\right)\dot{x}+x^{2n+1}+\sum_{j=0}^{n}a_{j}(t)x^{2j+1}=p(t),
\end{equation} which is reversible with respect to $G: (x, \dot{x}, t) \mapsto (-x, \dot{x}, -t),$ where either $p(t)=0$ or
$p(t)$ is odd. If we want to generalize \eqref{1.1} to a general reversible system $\ddot{x}+g(x,\dot x, t)=0$ where $g$ is finite smooth in each variable, then we need to construct a KAM theorem for reversible mapping of finite smoothness. This is one of aims that we write the present paper.

To that end, let $\mathbb{T}^{d}=(\mathbb{R}/2 \pi \mathbb{Z})^{d}, B(r)=\{y\in \mathbb{R}^{d}\mid |y|<r\}$ with $r>0,$ and let us consider a twist mapping
$$A_{0}:\;\;\left\{
              \begin{array}{ll}
                x_{1}=x+ \omega+ y, \\
                y_{1}=y,
              \end{array}
            \right.$$
where $(x, y)\in \mathbb{T}^{d}\times B(r_{0})$ with some $r_{0}>0$ is
a constant, as well as $\omega\in \mathbb{R}^{d}$ is called frequency of $A_{0}.$ It is clear
that $A_{0}$ possesses an invariant torus
$$\mathcal{J}_{0}:=\{x_{1}=x+\omega: x\in \mathbb{T}^{d}\} \times \{y_{1}=0\}.$$
We will prove that the invariant torus $\mathcal{J}_{0}$ is preserved undergoing a  small perturbation of finite smoothness, provided that $\omega$ is Diophantine.
More exactly, we have the following theorem:

\begin{theorem}\label{thm1}
Consider a mapping $A$ which is the perturbation of $A_{0}:$
$$A:\left\{
      \begin{array}{ll}
        x_{1}=x+\omega+y+f(x,y),\\
        y_{1}=y+g(x,y),
      \end{array}
    \right.
(x, y)\in \mathbb{T}^{d}\times B(r_{0}).
$$
Suppose that
\begin{itemize}
  \item[(A1)] $\omega\in DC(\kappa, \tau)$ with $0<\kappa <1,$ $\tau> d,$ that is,
  there exist constants $1>\kappa>0$ and $\tau>d$ such that
\begin{equation}\label{eq1}
\mid \langle k, \omega\rangle+ j\mid\geq \frac{\kappa}{|k|^{\tau}},\;\;\forall \;(k, j)\in \mathbb{Z}^{d}\times \mathbb{Z},\;\;k\neq 0.
\end{equation} (In order to the smoothness of perturbations $f$ and $g$ is sharp, we take $\tau=d+\frac{\mu}{100}$ with $0<\mu\ll1$.)
  \item[(A2)] Given $\ell=2d+1+\mu$ with $0<\mu\ll 1,$ and $f, g: \mathbb{T}^{d}\times B(r_{0})\rightarrow \mathbb{R}^{d}$ are $C^{\ell}$ and $C^{\ell+d},$ respectively, and
              $$ |f|_{C^{\ell}(\mathbb{T}^{d}\times B(r_{0}))}\leq \varepsilon,\;\;|g|_{C^{\ell+d}(\mathbb{T}^{d}\times B(r_{0}))}\leq \varepsilon.$$
  \item[(A3)] The mapping $A$ is reversible with respect to the involution  $G: (x, y)\mapsto (-x, y),$ that is,
$$AGA=G\;\;\mbox{on} \;\;\mathbb{T}^{d}\times B(r_{0}).$$
\end{itemize}
Then there exists $\varepsilon_{0}=\varepsilon_{0}(\tau, d, r_{0})>0$ such that for any $0<\varepsilon<\varepsilon_{0},$ the mapping $A$
has an invariant torus $\Gamma$
and the restriction of $A$ on $\Gamma$ is expressed by
$$A\mid _{\Gamma}: x\mapsto x+\omega.$$
\end{theorem}
\begin{theorem}\label{1.2}
Consider a system of non-autonomous differential equations
\begin{equation} \label{a}
(a):\;\;\left\{
  \begin{array}{ll}
    \dot{x}=\omega+y+f(x, y, t), \\
    \dot{y}=g(x, y, t),
  \end{array}
\right.\;\;(x, y, t)\in \mathbb{T}^{d}\times B(r_{0})\times \mathbb{T}:= D.
\end{equation}
Suppose that
\begin{itemize}
  \item [(a1)] $\omega\in DC(\kappa, \tau)$ with $0<\kappa <1,$ $\tau>d.$
  \item [(a2)] $f$, $g: \mathbb{T}^{d}\times B(r_{0})\times \mathbb{T}\rightarrow \mathbb{R}^{d}$ are $C^{\ell}$ and $C^{\ell+d},$ respectively, and
               $$|f|_{C^{\ell}(D)}\leq \varepsilon,\;\;|g|_{C^{\ell+d}(D)}\leq \varepsilon.$$
  \item [(a3)] The system \eqref{a} of differential equations is reversible with the involution $G: (x, y, t)\mapsto (-x, y, -t)$, that is,
               for any $(x, y, t)\in D,$
               \begin{eqnarray*}
               &f(-x, y, -t)=f(x, y, t),\\
               &g(-x, y, -t)=-g(x, y, t).
               \end{eqnarray*}
\end{itemize}
Then there exists $\varepsilon_{0}=\varepsilon_{0}(\tau, d, r_{0})>0$ such that for any $0<\varepsilon <\varepsilon_{0}$ there exists a coordinate
changes
\begin{eqnarray*}
\psi: \left\{
        \begin{array}{ll}
          x=\xi+u(\xi, \eta, t) \\
          y=\eta+v(\xi, \eta, t)
        \end{array}
      \right.
\end{eqnarray*}
such that the map $\psi$ restricted to $\{\xi=\omega t: t\in \mathbb{R}\}\times \{\eta=0\}\times \{t: t\in \mathbb{T}\}$ is a real $C^{0}$ embedding
into $\mathbb{T}^{d}\times \mathbb{R}^{d}$ of a rotational torus with frequency $\omega$ for the
system \eqref{a}.
\end{theorem}
\begin{rem}\label{rem0}
The first equation in \eqref{a} can be replaced by $\dot{x}=\omega+\alpha(y)+f(x, y, t)$ with $\partial_{y}\alpha(y)\geq C_{0}>0$ in the sense of positive definite matrix. In Theorem \ref{1.2}, the condition that $f$ and $g$ are $2\pi$ periodic in time $t$ can be generalized to that $f$ and $g$ are quasi-periodic with frequency $\tilde\omega\in \mathbb R^{\tilde d}$ ($\tilde d\ge 1$) in time $t$, if we take $\ell=2(d+\tilde d)-1+\mu$ and replace the condition $(a1)$ by $(\omega,\tilde\omega)\in\mathbb R^{d+\tilde d}\bigcap DC(\kappa,\tau)$ with $0<\kappa<1, \; \tau>d+\tilde d$.
\end{rem}

\begin{rem}\label{rem1}
For exact and area-preserving twist maps on annulus, it was proved by Herman in \cite{Her83} that unperturbed invariant curves can be destructed by $C^{3-\delta}$
$(0<\delta <1)$ arbitrarily small perturbation, and the unperturbed invariant curves can be preserved by $C^{3+\delta}$ sufficiently small
perturbation for some special Diophantine frequency $\omega$ of zero Lebesgue measure but infinitely many numbers.
 Recently,  Cheng-Wang \cite{Cheng} showed that for an integrable Hamiltonian
$H_{0}=\frac{1}{2}\sum_{i}^{d}y_{i}^{2}$ $(d\geq 2),$
any Lagrangian torus with a given unique rotation vector can be destructed by arbitrarily $C^{2d-\delta}$
small Hamiltonian perturbations. Checking those counter-examples above, it was seen that the results on
destructed invariant tori hold still true if an additional reversible condition is imposed on the symplectic mapping or Hamiltonian vector. So the  optimal smoothness of the reversible mapping $A$ depending periodically on time $t$ should be larger or equal to $2 d+1+\mu$. It is not clear whether the smoothness $3 d+1+\mu$  is optimal or this requirement is a shortcoming of our proof. It is worth to investigate further.
\end{rem}
\begin{rem}
The proof of Theorem \ref{thm1} is different from that of Moser's twist theorem. We recall that Moser invoked the smooth approximation technique from Nash's idea to
overcome the loss of regularity during Newton iteration. More exactly, Moser used a smoothness operator, say $S_{s}$, to decompose the perturbation vector field $(f, g)$
into $(f, g)=(S_{s}f+(1-S_{s})f, S_{s}g+(1-S_{s})g),$ where $S_{s}f$ and $S_{s}g$ are more smooth than $f$ and $g,$ and $(1-S_{s})f$ and $(1-S_{s})g$ are smaller than $f$
 and $g.$ Then he eliminated the perturbations $(S_{s}f, S_{s}g).$ An important fact is that $(S_{s}f, S_{s}g)$ is still symplectic if $(f, g)$ is symplectic. Unfortunately,
 when $(f, g)$ is reversible with respect to the involution $G,$ we do not know if $(S_{s}f, S_{s}g)$ is, too, reversible with respect to $G.$ So we could not transplant Moser's
 trick to deal with the reversible mapping of finite smoothness. In the present paper, we regard the reversible mapping $A$ in Theorem
 \ref{thm1} as the Poincare map of a reversible differential equation. And then we construct a KAM theorem (Theorem \ref{1.2}) for a reversible differential equation which
 is periodic in time. Then we proved Theorem \ref{thm1} by using Theorem \ref{1.2}.
\end{rem}

This paper is organized as follows: Section 2 we give out the approximation theorem of Jackson-Moser-Zhender. In Section 3, we give out iterative constants and iterative domains in Newton iteration. In Section 4, we give out the key iterative lemma (See Lemma \ref{lem5.1}).
Using Lemma \ref{lem5.1} further, we give proof of Theorems \ref{1.2} and \ref{thm1} .
In Section 5, we derive the homological equations and give out the estimates of the solutions of the homological equations to eliminates perturbations. In Section 6, we make the estimate of new perturbations. In Section 7, we give an application of the obtained Theorem \ref{thm1} to the Lagrange stability for reversible Duffing equation with finite smooth nonlinear perturbation.

\section{Approximation Lemma }

First we denote by $|\cdot|$ the norm of any finite dimensional Euclidean space.
Let $C^\mu(\mathbb R^m)$ for $0<\mu<1$ denote the space of bounded H\"older continuous functions $f:\; \mathbb R^m\to \mathbb R^n$ with the norm
\[|f|_{C^\mu}=\sup_{0<|x-y|<1}\frac{|f(x)-f(y)|}{|x-y|^\mu}+\sup_{x\in\mathbb R^m}|f(x)|.\]
If $\mu=0$ the $|f|_{C^\mu}$ denotes the sup-norm. For $\ell=k+\mu$ with $k\in\mathbb N$ and $0\le \mu<1$ we denote by $C^\ell(\mathbb R^m)$
the space of functions $f:\; \mathbb R^m\to \mathbb R^n$ with H\"older continuous partial derivatives $\p^\alpha\, f
\in\, C^\mu(\mathbb R^m)$ for all multi-indices $\alpha=(\alpha_1,...,\alpha_m)\in\mathbb N^m$ with $|\alpha|=\alpha_1+...+\alpha_m\le k
$. We define the norm
\[|f|_{C^\ell}:=\sum_{|\alpha|\le \ell} |\p^\alpha f|_{C^\mu}\] for $\mu=\ell-[\ell]<1$. In order to give an approximate lemma, we define the kernel function
\[K(x)=\frac{1}{(2\pi)^m}\int_{\mathbb R^m}\widehat{K}(\xi)e^{\mathbf{i}\, \<x,\xi\>}\, d\xi,\; x\in\mathbb C^m,\]
where $\widehat{K}(\xi)$ is a $C^\infty$ function with compact support, contained in the ball $|\xi|\le a$ with a constant $a>0$, that satisfies
\[\p^\alpha\,\widehat{K}(0)=\begin{cases} 1, & \text{if}\, \alpha=0,\\ 0, &  \text{if}\, \alpha\neq 0.\end{cases} \]
Then $K:\, \mathbb C^m\to \mathbb R^n$ is a real analytic function  with the property that for every $j>0$ and every $p>0$, there exists a constant $c_1=c_1(j,p)>0$ such that for all $\beta\in\mathbb N^m$ with $|\beta|\le j$,
\begin{equation}
\left|\p^\beta\, K(x+\mathbf{i} y) \right|\le c_{1} (1+|x|)^{-p}e^{a|y|},\;\; x,y\in\mathbb R^m.\label{y1}
\end{equation}


\begin{lemma}\label{lem3}(Jackson-Moser-Zehnder)
There is a family of convolution operators
 \be (S_{s}F)(x)=s^{-m}\int_{\mathbb{R}^{m}}K(s^{-1}(x-y))F(y)dy,\;\;0<s\leq 1,\;\;\forall\;F\in C^{0}(\mathbb{R}^{m})\label{y2}\ee from $C^{0}(\mathbb{R}^{m})$ into the linear space of entire (vector) functions on $\mathbb{C}^{m}$ such that for
every $\ell>0$ there exist a constant $c=c(\ell)>0$ with the following properties: If $F\in C^{\ell}(\mathbb{R}^{m}),$
 then for $|\alpha|\leq \ell$ and $|\mathrm{Im} x|\leq s,$
\begin{equation}\label{2.11}
|\partial^{\alpha}(S_{s}F)(x)-\sum_{|\beta|\leq \ell-|\alpha|}\partial^{\alpha+\beta}F(\mathrm{Re} x)({\bf{i}}\, \mathrm{Im} x)^{\beta}/\beta!|\leq c\,|F|_{C^{\ell}}s^{\ell-|\alpha|}.
\end{equation}
Moreover, in the real case\begin{eqnarray}
&& |S_{s}F-F|_{C^{p}}\leq c |F|_{C^{\ell}}s^{\ell-p},\;\;p\leq \ell,\label{2.13-1}\\
&& |S_{s}F|_{C^{p}}\leq c|F|_{C^{\ell}}s^{\ell-p},\;\;p\leq \ell.\label{2.13}
\end{eqnarray}
Finally, if $F$ is periodic in some variables then so are the approximating functions $S_{s}F$ in the same variables.
\end{lemma}
\begin{rem}\label{rem**}
 Moreover we point out that from \eqref{2.13} one can easily deduce the following well-known
convexity estimates which will be  used later on
\begin{eqnarray}
&&|f|_{C^{\alpha}}^{l-k}\leq c|f|_{C^{k}}^{l-\alpha}|f|_{C^{l}}^{\alpha-k},\;\;k\leq \alpha\leq l,\label{5.4-1}\\
&& |f\cdot g|_{C^{s}}\leq c(|f|_{C^{s}}|g|_{C^{0}}+|f|_{C^{0}}|g|_{C^{s}}),\;\;s\geq 0.\label{5.5-1}
\end{eqnarray}
See \cite{Salamon2004, Zehnder1976} for the proofs of Lemma \ref{lem3} and the inequalities \eqref{5.4-1} and \eqref{5.5-1}.
\end{rem}
\begin{rem}\label{rem**y} From the definition of the operator $S_s$, we clearly have
\be \sup_{x,y\in\mathbb R^m,|y|\le s}\, \left|S_s\, F(x+\mathbf{i}\, y) \right|\le C |F|_{C^0}.\label{3.x}\ee
In fact, by the definition of  $S_s$, we have that for any $x,y\in\mathbb R^m$ with $|y|\le s$,
\begin{eqnarray}\left|S_s\, F(x+\mathbf{i}\, y) \right|&=&\left|s^{-m}\int_{\mathbb R^m}K(s^{-1}(x+\mathbf{i}\, y-z))F(z)\, d\,z \right|
\\
&=& \left|\int_{\mathbb R^m}K(\mathbf{i}\,s^{-1} y+\xi)F(x-s \xi)\, d\,\xi\right| \\
&\le & |F|_{C^0}\, \int_{\mathbb R^m}\left|K(\mathbf{i}\,s^{-1} y+\xi)\right|\, d\,\xi\\
&\le & C \, |F|_{C^0},
\end{eqnarray}
where we used \eqref{y1} in the last inequality. The following lemma shows that the operator $S_{s}$ commutes with the involution map
$G: (x, y, t)\mapsto (-x, y, -t).$
\end{rem}

\begin{lemma}\label{lem3.2}
Let $m=2d+1$ in Lemma \ref{lem3}. Then

(1)when $F(-x, y, -t)=-F(x, y, t),$ for $\forall (x, y, t)\in\mathbb{T}^{d}\times B(r)\times \mathbb{T}$ with $r>0,$ we have
              $$S_{s} F(-x, y, -t)=- S_{s}F(x, y, t).$$

(2) when $F(-x, y, -t)=F(x, y, t),$ for $\forall (x, y, t)\in\mathbb{T}^{d}\times B(r)\times \mathbb{T}$ with $r>0,$ we have
              $$S_{s} F(-x, y, -t)=S_{s}F(x, y, t).$$

\end{lemma}
\begin{proof}
Let $\xi=(\xi_{1}, \xi_{2}, t)\in \mathbb{R}^{d}\times \mathbb{R}^{d}\times\mathbb{R}=\mathbb{R}^{m}.$ Clearly we can choose the kernel function
$\widehat{K}(\xi)$ such that
$$\widehat{K}(-\xi_{1}, \xi_{2}, -t)=\widehat{K}(\xi_{1}, \xi_{2}, t),\;\;\forall \xi\in\mathbb{R}^{m}.$$
It follows that
$$K(-x, y, -t)=(-1)^{d+1}K(x, y, t),\;\;\forall (x, y, t)\in \mathbb{R}^{m}.$$
By the definition of $S_{s},$
$$S_{s}F(x, y, t)=s^{-(2d+1)}\int_{\mathbb{R}^{d}\times \mathbb{R}^{d}\times \mathbb{R}}
K(s^{-1}(x-\tilde{x}, y-\tilde{y}, t-\tilde{t}))F(\tilde{x}, \tilde{y}, \tilde{t})d \tilde{x} d \tilde{y} d \tilde{t}.$$
So
\begin{eqnarray*}
S_{s}F(-x, y, -t)&=&s^{-(2d+1)}\int_{\mathbb{R}^{d}\times \mathbb{R}^{d}\times \mathbb{R}}
K(s^{-1}(-x-\tilde{x}, y-\tilde{y}, -t-\tilde{t}))F(\tilde{x}, \tilde{y}, \tilde{t})d \tilde{x} d \tilde{y} d \tilde{t}\\
&=& s^{-(2d+1)}(-1)^{d+1}\int_{\mathbb{R}^{d}\times \mathbb{R}^{d}\times \mathbb{R}}
K(s^{-1}(x+\tilde{x}, y-\tilde{y}, t+\tilde{t}))F(\tilde{x}, \tilde{y}, \tilde{t})d \tilde{x} d \tilde{y} d \tilde{t}\\
&=&s^{-(2d+1)}\int_{\mathbb{R}^{d}\times \mathbb{R}^{d}\times \mathbb{R}}
K(s^{-1}(x-{x}^{*}, y-\tilde{y}, t-{t}^{*}))F(-{x}^{*}, \tilde{y}, -{t}^{*})d {x}^{*} d \tilde{y} d {t}^{*}\\
&=& \mp s^{-(2d+1)}\int_{\mathbb{R}^{d}\times \mathbb{R}^{d}\times \mathbb{R}}
K(s^{-1}(x-\tilde{x}, y-\tilde{y}, t-\tilde{t}))F(\tilde{x}, \tilde{y}, \tilde{t})d \tilde{x} d \tilde{y} d \tilde{t}\\
&=& \mp S_{s}F(x, y, t),
\end{eqnarray*}
where $\mp=-$ for case (1), $\mp=+$ for case (2).
\end{proof}

Consider a $\mathbb{R}^{n}-$ valued function $F: D\rightarrow \mathbb{R}^{n}$ with
$$|F|_{C}^{\ell}(D)\leq \varepsilon.$$ Recall $D=\mathbb{T}^{d}\times B(r_{0})\times \mathbb{T}.$ By Whitney's extension theorem, we can
find a $\mathbb{R}^{n}-$ valued function $\widetilde{F}: \mathbb{T}^{d}\times \mathbb{R}^{d}\times \mathbb{T}\rightarrow \mathbb{R}^{n}$
such that
$\widetilde{F}\mid_{D}=F$ \;\;(i. e. $\widetilde{F}$ is the extension of $F$) and
$$|\widetilde{F}|_{C^{|\alpha|}(\mathbb{T}^{d}\times \mathbb{R}^{d}\times \mathbb{T})}\leq C_{\alpha}\mid F\mid _{C^{|\alpha|}(D)},
\;\;\forall \alpha\in \mathbb{Z}^{d}_{+},\;|\alpha|\leq \ell,$$
where $C_{\alpha}$ is a constant depends only $\ell$ and $d.$

Let $z=(x, y, t)$ for brevity, define, for $\forall s>0,$
$$(S_{s} \widetilde{F})(z)=s^{-(2d+1)}\int_{\mathbb{T}^{d}\times \mathbb{R}^{d}\times \mathbb{T}}K(s^{-1}(z-\tilde{z}))\widetilde{F}(\tilde{z})d\tilde{z}.$$
Let $\mathbb{T}^{d}_{s}=\{\phi\in (\mathbb{C}/ 2\pi \mathbb{Z})^{d}: |Im \phi|<s\},$ $\mathbb{R}^{d}_{s}=\{x\in \mathbb{C}^{d}\mid |Im x|<s\}.$
Fix a sequence of fast decreasing numbers $s_{\nu} \downarrow 0,$ $\nu\in \mathbb{Z}_{+}$ and $s_{0}\leq 1/2.$ Let
$$F^{(\nu)}(z)=(S_{s_{\nu}}\widetilde{F})(z),\;\;\nu\geq 0.$$
Then $F^{(\nu)}$'s $(\nu\geq 0)$ are entire functions in $\mathbb{C}^{2d+1},$ in particular, which obey the following properties.

\begin{itemize}
  \item [(1)] $F^{(\nu)}$'s $(\nu\geq 0)$ are real analytic \footnote{that is, $F^{\nu}(z)$ is analytic in
  $\mathbb{T}^{d}_{s_{\nu}}\times\mathbb{R}^{d}_{s_{\nu}}\times\mathbb{T}_{s_{\nu}}$, and is real when $z$ is real.}  on the complex domain $\mathbb{T}^{d}_{s_{\nu}}\times \mathbb{R}^{d}_{s_{\nu}}\times \mathbb{T}_{s_{\nu}}:=D_{s_{\nu}};$
  \item [(2)] The sequence of functions $F^{(\nu)}(z)$ satisfies the bounds
\begin{eqnarray}
&&\sup_{z\in D_{s_{\nu}}}|F^{(\nu)}(z)-F(z)|\leq C |F|_{C^{\ell}(D)}s_{\nu}^{\ell},\label{09.1}\\
&& \sup_{z\in D_{s_{\nu+1}}}|F^{(\nu+1)}(z)-F^{(\nu)}(z)|\leq C |F|_{C^{\ell}(D)}s_{\nu}^{\ell},\label{09.2}
\end{eqnarray}
where constants $C=C(d, \ell)$ depend on only $d$ and $\ell;$
  \item [(3)] The first approximate $F^{(0)}(z)=(S_{s_{0}}\widetilde{F})(z)$ is ``small" with respect to $F.$ Precisely,
             \begin{equation}\label{09.3}
             |F^{(0)}(z)|\leq C|F|_{C^{\ell}(D)},\;\;\forall z\in D_{s_{0}},
             \end{equation}
where constant $C=C(d, \ell)$ is independent of $s_{0}.$
  \item [(4)] From Lemma \ref{lem3}, we have that
  \be\label{09.4}F(z)=F^{(0)}(z)+\sum_{\nu=0}^{\infty}(F^{(\nu+1)}(z)-F^{(\nu)}(z)),\;\;\forall z\in D.\ee
  Let
 \be \label{09.5}F_{0}(z)=F^{(0)}(z),\;\;F_{\nu+1}(z)=F^{\nu+1}(z)-F^{\nu}(z).\ee
 Then
 \be\label{09.6}F(z)=\sum_{\nu=0}^{\infty}F_{\nu}(z),\;\;\forall z\in D.\ee
\end{itemize}

By Lemma\ref{lem3.2}, we have
\begin{eqnarray}
&& F_{\nu}(-x, y, -t)=-F_{\nu}(x, y, t),\;\;\text{if}\;\;F(-x, y, -t)=-F(x, y, t),\label{09.7}\\
&& F_{\nu}(-x, y, -t)=F_{\nu}(x, y, t),\;\;\text{if}\;\;F(-x, y, -t)=F(x, y, t).\label{09.8}
\end{eqnarray}


\section{Iterative constants}
\begin{itemize}
\item Given constant $\mu$ with $0<\mu\ll 1$, and let $\tau=d+\frac{\mu}{100}$, $\widetilde{\mu}=\frac{\mu}{100(2\tau+1+\mu)};$
 \item $\ell=2 d+1+\mu\;;$
  \item $\varepsilon_{0}=\varepsilon,$ \footnote{We hope that the readers are able to distinguish this $\varepsilon_{0}$ with that in Theorems 1.1 and 1.2.} $\varepsilon_{\nu}=\varepsilon^{(1+{\widetilde{\mu}})^{\nu}},$
  $\nu=0, 1, 2, \cdots,$ which measures the size of perturbation at $\nu-$th step of Newton iteration;
  \item   $s_{\nu}=\varepsilon_{\nu}^{1/\ell},\;\;\nu=0, 1, 2, \cdots$, which measures the width of angle variable in analytic approximation;
\item $r_{\nu}=s_{\nu}^{d+1+\frac{\mu}{10}},\;\;\nu=0, 1, 2, \cdots$, which measures the size of action variable in analytic approximation;
  \item
$s^{(j)}_{\nu}=s_{\nu}-\frac{j}{100\, \ell}(s_{\nu}-s_{\nu+1}),\; j=1,2,...,100\, \ell$, which are bridges between $s_{\nu}$ and $s_{\nu+1}$;
 \item
$r^{(j)}_{\nu}=r_{\nu}-\frac{j}{100\, \ell}(r_{\nu}-r_{\nu+1}),\; j=1,2,...,100\, \ell$, which are bridges between $r_{\nu}$ and $r_{\nu+1}$;
\item
$B_{\mathbb{C}}(r)=\{y\in \mathbb{C}^{d}: |y|\leq r\},$ for $\forall r\geq 0;$
  \item $D(s, r)=\mathbb{T}^{d}_{s}\times B_{\mathbb{C}}(r)\times \mathbb{T}_{s},\;\;\forall r\geq 0, s\geq 0.$
  \item
  For a $\mathbb{C}^{n}-$ valued function $F(x, y, t)$ analytic in $D(s, r),$ denote
  $$||F||_{s, r}=\sup_{z=(x, y, t)\in D(s, r)}|F(z)|,$$ here (and other places) $|\cdot|$ is Euclidean norm.
\end{itemize}

\section{Iterative Lemma}\label{lemma4.1}

Let us return to function $f=f(x, y, t),$ $g=g(x, y, t)$ in Theorem \ref{1.2}. Let $z=(x, y,t)$ for brevity.
With the above preparation, we can rewrite equation \eqref{a} in Theorem \ref{1.2} as follows:
\be \label{92.0}\dot{x}=\omega+y+\sum_{\nu=0}^{\infty}f_{\nu}(z),\;\dot{y}=\sum_{\nu=0}^{\infty}g_{\nu}(z),\ee
where
\be\label{92.1} f_{\nu}, g_{\nu}: \mathbb{T}^{d}_{s_{\nu}}\times \mathbb{R}^{d}_{s_{\nu}}\times\mathbb{T}_{s_{\nu}}\rightarrow\mathbb{C}^{n}\ee
are real analytic, and
\be\label{92.2}\parallel f_{\nu}\parallel_{s_{\nu}, r_{\nu}}\leq C \varepsilon_{\nu},\;\;\parallel g_{\nu}\parallel_{s_{\nu},r_{\nu}}\leq C\varepsilon_{\nu}s_{\nu}^{d}\ee
and for any $(x,y,t)\in D(s_{\nu}, r_{\nu})$
\be\label{92.3} f_{\nu}(-x, y, -t)=f_{\nu}(x, y, t),\;\;g_{\nu}(-x, y, -t)=-g_{\nu}(x, y, t).\ee
The basic idea in KAM theory is to kill the perturbations $f$ and $g$ by Newton iteration. The procedure of the iteration is as follows:

$1^{st}$ step: to search for a involution map $\Phi_{0}$ (which keeps the involution $G: (x, y, t)\mapsto (-x, y, -t)$ unchanged) such that the analytic vector-field
$${(a)_{0}:\;\; (\omega+y+f_{0}, g_{0})}$$ is changed by $\Phi_{0}$ into
\begin{eqnarray*}
(a)_{0}\circ \Phi_{0}&=&(\omega+y+f_{0}, g_{0})\circ \Phi_{0}\\
&=&(\omega+y+f^{0}_{1}, g_{1}^{0}):=(a)_{1},
\end{eqnarray*}
where $f^{0}_{1}=O (\varepsilon_{1})$ and $g^{1}_{0}=O(\varepsilon_{1}s^{d}_{1}).$

$2^{nd}$ step: to search for a involution map $\Phi_{1}$ such that $(a)_{1}+(f_{1}, g_{1})$ is changed into
$(\omega+y+f_{2}^{0}, g_{2}^{0}),$
where $f_{2}^{0}=O(\varepsilon_{2},)$ $g_{2}^{0}=O(\varepsilon_{2}s_{2}^{d})$.
The combination of steps 1 and 2 implies that $(\omega+y+f_{0}+f_{1}, g_{0}+g_{1})\circ \Phi_{0}\circ\Phi_{1}=(\omega+y+f^{0}_{2}, g^{0}_{2}).$
Repeating the above procedure, at $m+1^{th}$ step, we have that $(a)_{m}+(f_{m}, g_{m})$ is changed by $\Phi_{m}$ into
$(\omega+y+f^{0}_{m+1}, g_{m+1}^{0}),$
where $f_{m+1}^{0}=O(\varepsilon_{m+1}),$ $g_{m+1}^{0}=O(\varepsilon_{m+1}s_{m+1}^{d}).$
That is, $(\omega+y+\sum_{j=0}^{m}f_{j}, \sum_{j=0}^{m}g_{j})\circ \Phi^{(m)}=(\omega+y+f^{0}_{m+1}, g_{m+1}^{0}),$
where $\Phi^{(m)}=\Phi_{0}\circ\Phi_{1}\circ\cdots \circ\Phi_{m}.$
Finally letting $m\rightarrow \infty,$ and letting
$$\Phi^{\infty}:=\lim_{m\rightarrow \infty}\Phi^{(m)},$$
we have
\begin{eqnarray*}
&&(\omega+y+f,g)\circ \Phi^{\infty}\\
&&=\lim_{m\rightarrow \infty} (\omega+y+\sum_{j=0}^{m}f_{j}, \sum_{j=0}^{m}g_{j})\circ \Phi^{(m)}\\
&&=\lim_{m\rightarrow \infty} (\omega+y+f_{m+1}^{0}, g_{m+1}^{0})\\
&&=\lim_{m\rightarrow \infty} (\omega+y+O(\varepsilon_{m+1}), O(\varepsilon_{m+1}s_{m+1}^{d}))\\
&&=(\omega+y, 0).
\end{eqnarray*}
From this, we see that $(\Phi^{\infty})^{-1}(\{x=\omega t\}\times \{y=0\}\times \{t: t\in \mathbb{T}\})$ is an invariant torus of the original vector-field.
This iterative procedure can be found in \cite{Chierchia-Qian}. The following iterative Lemma is a
materialization of the above iterative procedure.
\begin{lemma}\label{lem5.1}(Iterative Lemma)
Let $\omega\in DC(\kappa, \tau).$
Assume that we have $m$ coordinate changes $\Phi_{0}=\Psi_{0}^{-1},\cdots, \Phi_{m-1}=\Psi_{m-1}^{-1},$ which obey
$$\Psi_{j}: D(s_{j}, r_{j})\rightarrow D(s_{j-1}, r_{j-1})\;\;(j=0, 1, \cdots, m-1)$$ of the form
$$\Psi_{j}: x=\tilde{x}+u_{j}(\tilde{x}, \tilde{y}, t),\;y=\tilde{y}+v_{j}(\tilde{x}, \tilde{y}, t),\;(j=0, 1, \cdots, m-1)$$ and
$u_{j},$ $v_{j}$ are real for real arguments and analytic in each argument with estimates
\be\label{y5.1}\parallel u_{j}\parallel_{s_{j}, r_{j}}\leq C\varepsilon_{j}s_{j}^{-d}, \parallel v_{j}\parallel_{s_{j}, r_{j}}\leq C\varepsilon_{j}\;\;(j=0, 1, \cdots, m-1)\ee
such that the system of equations
$$(a)^{(m-1)}: \left\{
                 \begin{array}{ll}
                    \dot{\tilde{x}}=\omega+\tilde{y}+\sum_{j=0}^{m-1}f_{j}(\tilde{x}, \tilde{y}, t),\\
                    \dot{\tilde{y}}=\sum_{j=0}^{m-1}g_{j}(\tilde{x}, \tilde{y}, t)
                 \end{array}
               \right.
$$
is changed by $\Phi^{(m-1)}=\Phi_{0}\circ \cdots\circ\Phi_{m-1}$ into
$$(a)_{*}^{(m-1)}: \left\{
                 \begin{array}{ll}
                    \dot{{x}}=\omega+{y}+f^{0}_{m}({x}, {y}, t),\\
                    \dot{{y}}=g^{0}_{m}({x}, {y}, t),
                 \end{array}
               \right.
$$ where $f_{m}^{0},$ $g_{m}^{0}$ obey

\begin{itemize}
  \item [$(1)_{m}$] The functions $f_{m}^{0}$, $g_{m}^{0}$ are real for real arguments;
  \item [$(2)_{m}$] The functions $f_{m}^{0}$, $g_{m}^{0}$ are analytic in $D(s_{m}, r_{m})$ with estimates
\be\label{y5.1} \parallel f_{m}^{0}\parallel_{s_{m}, r_{m}}\leq C \varepsilon_{m},\;\;\parallel g_{m}^{0}\parallel_{s_{m}, r_{m}}\leq C \varepsilon_{m}s_{m}^{d};\ee
  \item [$(3)_{m}$] The functions $f_{m}^{0}$, $g_{m}^{0}$ is reversible with respect to involution map: $G: (x, y, t)\mapsto (-x, y, -t),$ that is,
\be\label{y5.2} f_{m}^{0}(-x, y, -t)=f_{m}^{0}(x, y, t),\;\;g_{m}^{0}(-x, y, -t)=-g_{m}(x, y, t).\ee
\end{itemize}
Then there is a coordinate change $\Phi_{m}=\Psi^{-1}_{m}:$
$$\Psi_{m}: D(s_{m+1}, r_{m+1})\rightarrow D(s_{m}, r_{m})$$ of the form
\be\label{y5.3} \Psi_{m}: \xi=x+u_{m}(x, y, t),\;\;\eta=y+v_{m}(x, y, t)\ee
and
\be\label{y5.4} \parallel u_{m}\parallel_{s_{m}, r_{m}}\leq C \varepsilon_{m}s_{m}^{-d},\;\;\parallel v_{m}\parallel_{s_{m}, r_{m}}\leq C\varepsilon_{m}\ee
such that $\Psi_{m},$ which is reversible with respect to $G: (x, y, t)\mapsto (-x, y, -t),$ changes the modified equations:
\be\label{y5.5}
(a_{j})^{*}:  \left\{
                 \begin{array}{ll}
                   \dot{x}=\omega+y+f_{m}^{0}(x, y, t)+f_{m}(x, y, t)\\
                   \dot{y}=g_{m}^{0}(x, y, t)+g_{m}(x, y, t)
                 \end{array}
               \right.
\ee
into
\be \label{eq4.11}(a)^{(m)}_{*}: \left\{
                     \begin{array}{ll}
                        \dot{\xi}=\omega+\eta+f_{m+1}^{0}(\xi, \eta, t),\\
                        \dot{\eta}=g_{m+1}^{0}(\xi, \eta, t),
                     \end{array}
                   \right.
\ee
where $f_{m+1}^{0}$ and $g_{m+1}^{0}$ obey the conditions $(1)_{m},$ $(2)_{m}$ and $(3)_{m}$ by replacing $m$ by $m+1.$
In other words, $\Phi^{(m)}:=\Phi^{(m-1)}\circ\Phi_{m}$ changes
\be\label{eq5.10}
\left\{
  \begin{array}{ll}
     \dot{x}=\omega+y+\sum_{j=0}^{m}f_{j}(x, y, t)\\
     \dot{y}=\sum_{j=0}^{m}g_{j}(x, y, t)
  \end{array}
\right.
\ee
into $(a)^{(m)}_{*}.$

\end{lemma}


{\bf Proof of Theorem 1.2} We see that
$$\Psi^{\infty}=\lim_{m\rightarrow \infty}\Psi_{1}\circ \cdots\circ\Psi_{m}: D(0, 0)\rightarrow D(s_{0}, r_{0})\subset D.$$
The proof for the existence of limit $\Psi^{\infty}$ is now standard. We omit the detail. See Moser \cite{Moser1962} for example.
Let  $\Phi^{\infty}=\lim_{m\rightarrow \infty} \Phi^{(m)}.$ Thus $\Psi^{\infty}=(\Phi^{\infty})^{-1}$ and $\Psi^{\infty}(\{\omega t\}\times \{y=0\}\times \{t: t\in \mathbb{T}\})$ is
a $C^{0}$ embedding torus of the original equations \eqref{a}.

{\bf Proof of Theorem 1.1}
By Proposition 4.5 of  Sevryuk \cite{Sevryuk1986}, we know that the map $A$ in Theorem 1.1 can be regarded as the time-1 map ( Poincare map ) of equation \eqref{a} in Theorem 1.2. The proof is completed by Theorem 1.2.

\section{Derivation of homological equation}
Let us recall \eqref{y5.5}. Let
\begin{eqnarray}
&&f_{m}^{(m)}(x, y, t)=f_{m}^{0}(x, y, t)+f_{m}(x, y, t),\label{18.1}\\
&&g_{m}^{(m)}(x, y, t)=g_{m}^{0}(x, y, t)+g_{m}(x, y, t).\label{18.2}
\end{eqnarray}
By the conditons $(1)_{m},$ $(2)_{m}$ and $(3)_{m}$ in the iterative lemma, and observing \eqref{92.1}, \eqref{92.2} and \eqref{92.3}, we have
\begin{itemize}
  \item [$(i)_{m}$] the functions $f_{m}^{(m)},$ $g_{m}^{(m)}$ are real for real arguments;
  \item [$(ii)_{m}$] the functions $f_{m}^{(m)},$ $g_{m}^{(m)}$ are analytic in $D(s_{m}, r_{m}):$
  \be\label{eq*}\parallel f_{m}^{(m)}\parallel_{s_{m}, r_{m}}\leq C \varepsilon_{m},\;\;\parallel g_{m}^{(m)}\parallel_{s_{m}, r_{m}}\leq C \varepsilon_{m}s_{m}^{d};\ee
  \item [$(iii)_{m}$] the functions $f_{m}^{(m)},$ $g_{m}^{(m)}$ are reversible with respect to $G: (x, y, t)\mapsto (-x, y, -t),$ that is
  \begin{eqnarray}
&&f_{m}^{(m)}(-x, y, -t)=f_{m}^{(m)}(x, y, t),\label{eq6.17}\\
&&g_{m}^{(m)}(-x, y, -t)=-g_{m}^{(m)}(x, y, t).\label{eq6.18}
\end{eqnarray}
\end{itemize}

Consider a map $\Phi=\Phi_{m}$ of the form
\be\label{eq6.1} \Phi=\Phi_{m}: \left\{
                                \begin{array}{ll}
                                  x=\xi+U(\xi, \eta, t)\\
                                  y=\eta+V(\xi, \eta, t)
                                \end{array}
                              \right.
\ee
and its inverse $\Phi^{-1}$ is of the form
\be\label{eq6.2}\Psi=\Phi^{-1}: \left\{
                           \begin{array}{ll}
                              \xi=x+u(x, y, t),\\
                              \eta=y+v(x,y,t),
                           \end{array}
                         \right.
\ee
where $u, v, U, V$ will be specified. Inserting \eqref{eq6.2} into equation $(a_{j})^{*}$ with $j=m-1$ (i.e. \eqref{y5.5}), and
noting \eqref{18.1} and \eqref{18.2}, we have that
\begin{eqnarray}
\dot{\xi}&=&\omega+\eta\\
&+& \omega\cdot \partial_{x}u+\partial_{t} u+f_{m}^{(m)}(x, y, t)-v\label{eq6.3}\\
&+& \partial_{y} u\cdot g_{m}^{(m)}(x, y, t)+\partial_{x}u\cdot f_{m}^{(m)}(x, y, t)\label{eq6.4}\\
&+& \partial_{x}u\cdot y,\label{eq6.6}
\end{eqnarray}
\begin{eqnarray}
\dot{\eta}&=&\omega\cdot\partial_{x}v+\partial_{t}v+g_{m}^{(m)}(x,y,t)\label{eq6.8}\\
&+& \partial_{y}v\cdot g_{m}^{(m)}(x, y, t)+\partial_{x}v\cdot f_{m}^{(m)}(x, y, t)\label{eq6.9}\\
&+& \partial_{x}v\cdot y \label{eq6.11}
\end{eqnarray}
and where $u=u(x, y, t),$ $v=v(x, y, t),$ $\omega\cdot\partial_{x}=\Sigma_{j=1}^{d}\omega_{j}\partial_{x_{j}},$ and
$x=\xi+U(\xi, \eta, t),$ $y=\eta+V(\xi, \eta, t)$ will be implicity defined by \eqref{eq6.2}.

Letting \eqref{eq6.3}=0 and \eqref{eq6.8}=0, we derive homological equations:
\begin{equation}\label{eq6.15}
\omega\cdot \partial_{x}u+\partial_{t}u-v+f_{m}^{(m)}(x, y, t)=0
\end{equation}
and
\begin{equation}\label{eq6.16}
\omega\cdot \partial_{x}v+\partial_{t}v+g_{m}^{(m)}(x, y, t)=0.
\end{equation}

Let $\widehat{g}_{m}^{(m)}(k, l, y)$ is the $(k, l)-$ Fourier coefficient of $g_{m}^{(m)}(x, y, t),$ with respect to
variable $(x, t)$, that is,
$$\widehat{g}_{m}^{(m)}(k, l, y)=\frac{1}{(2\pi)^{d+1}}\int_{0}^{2\pi}\cdots\int_{0}^{2\pi} g_{m}^{(m)}(x, y, t)e^{-\sqrt{-1}(\langle k, x\rangle+l t)}dx dt,$$
where $k\in \mathbb{Z}^{d}, l\in \mathbb{Z}.$ Similarly, we can define $\widehat{f}_{m}^{(m)}(k, l, y),$ etc.

By \eqref{eq6.18}, we have
\be\label{eq6.19} \widehat{g}_{m}^{(m)}(0, 0, y)=0,\;\;y\in {B}_{\mathbb{C}}(r_{m}).\ee
By passing to Fourier coefficients, homological equation \eqref{eq6.16} reads
\be\label{eq6.20} \sqrt{-1}(\langle k, \omega\rangle+l)\,\widehat{v}(k, l, y)=-\widehat{g}^{(m)}_{m}(k, l, y),\ee
where $(k, l)\in \mathbb{Z}^{d}\times \mathbb{Z}\setminus \{(0, 0)\},$ $y\in B_{\mathbb{C}}(r_{m}).$ So we have
\be\label{eq6.21}\widehat{v}(k,l, y)=\sqrt{-1}\frac{\widehat{g}_{m}^{(m)}(k, l, y)}{\langle k, \omega\rangle+l}.\ee
We notice that when $(k, l)=(0, 0),$ $\langle k, \omega\rangle+l=0$ and $\widehat{g}_{m}^{(m)}(0, 0, y)=0.$ Thus we
have a freedom to choose $\widehat{v}(0, 0, y)$ in \eqref{eq6.20}.
\begin{lemma}\label{ly5.1} \cite{Russmann1975, Russmann1983} Assume $\omega$ satisfies
$$|\langle k, \omega\rangle+j|\geq \kappa/|k|^{\tau},\;\;\forall (k, j)\in\mathbb{Z}^{d}\times\mathbb{Z}\setminus\{(0,0)\}.$$
Then the inequalities

$$\sum_{|k|\leq n}|\langle k, \omega\rangle+j|^{-2}\leq C\kappa^{-2}n^{2\tau},\;\;|k|=|k_{1}|+\cdots+|k_{d}|$$
hold for $n=1, 2, \cdots.$
\end{lemma}
\begin{proof}
The proof for $d=1$ is given in \cite{Russmann1975,Russmann1983}. For $d>1,$ the proof goes well. For the convenience of the readers, we copy
the proof from \cite{Russmann1975,Russmann1983} with a minor modification.
If we numerate the numbers of the set
$$\{\<k, \omega\>+l\mid |k|\leq n, l\in \mathbb{Z}\},$$
according to their natural order
\begin{eqnarray*}
&&\cdots < d_{-2}< d_{-1}< 0 < d_{1}< d_{2}<\cdots,\\
&& d_{j}=\<k_{j}, \omega\>- l_{j}, j=\pm 1, \pm2, \cdots.
\end{eqnarray*}
According to $\omega\in D(\kappa, \tau),$ $d_{1}\geq \kappa/ n^{\tau}$, we obtain
$$d_{j+1}-d_{j}=|\<k_{j+1}-k_{j}, \omega\>+ (l_{j}-l_{j+1})|\geq \kappa/2^{\tau}n^{\tau},\;\;j=1, 2, \cdots.$$
Thus $$d_{j}\geq j\kappa (2n)^{-\tau},\;\;j=1, 2, \cdots.$$
It follows that
$$\sum_{j=1}^{\infty}d_{j}^{-2}\leq \kappa^{-2}(2n)^{2\tau}\sum_{j=1}^{\infty}\frac{1}{j^{2}}\leq C\kappa^{-2}n^{2\tau}.$$
In the same way, we have
$$\sum_{j=-\infty}^{-1}d_{j}^{-2}\leq C\kappa^{-2}n^{2\tau}.$$
Consequently, we have
$$\sum_{|k|\leq n}|\langle k, \omega\rangle+l|^{-2}\leq C \kappa^{-2}n^{2\tau}.$$\end{proof}
We are now in position to estimate $v(x, y ,t)$. First, by Parseval's identity
\be\label{li-1}\sup_{y\in B_{\mathbb{C}}(r_{m})}\sum_{(k, l)\in \mathbb{Z}^{d+1}\setminus \{0\}}\mid \widehat{g}_{m}^{(m)}(k, l, y)\mid^{2}e^{2(|k|+|l|)s_{m}}
\leq C\parallel g_{m}^{(m)}\parallel^{2}_{s_{m}, r_{m}},\ee
where $C=C(\kappa,\tau)$ depends on $\kappa$ and $ \tau.$
Following \cite{Nash1956,Piao2008}, let
$$G_{n}(y)=\sum_{\begin{array}{c}
                     (k, l)\in \mathbb{Z}^{d+1} \\
                     1\leq |k|+|l|\leq n \\
                   \end{array}}\mid\widehat{g}^{(m)}_{m}(k, l, y)\mid\frac{e^{(|k|+|l|)s_{m}}}{|\langle k, \omega\rangle+l|},\;\;(n=1, 2, \cdots).$$
Then by Cauchy-Schwarz inequality and Lemma \ref{ly5.1}, we get
\begin{eqnarray*}
G_{n}(y)&\leq & C\sqrt{\sum_{(k,l)\in\mathbb{Z}^{d+1}}\mid\widehat{g}_{m}^{(m)}(k, l,y)\mid^{2}e^{2(|k|+|l|)s_{m}}}\sqrt{\sum_{1\leq |k|+|l|\leq n}
\frac{1}{\mid \langle k, \omega\rangle+l\mid^{2}}}\\
&\leq & C\parallel g_{m}^{(m)}\parallel_{s_{m}, r_{m}}\frac{|n|^{\tau}}{\kappa}.
\end{eqnarray*}
Letting $\tilde b_m=\frac{1}{200\ell}(s_m-s_{m+1})\ge \frac{1}{400\ell} s_m$ and letting $G_{0}(y)=0,$ we obtain by means of Abel's partial summation, for any $N\gg 1,$
\begin{eqnarray*}
&&\sum_{0<|k|+|l|\leq N}\mid \widehat{g}_{m}^{(m)}(k, l, y)\mid\frac{e^{(|k|+|l|)(s_{m}-\tilde b_m)}}{\mid \langle k, \omega\rangle+l\mid}\\
&=&(1-e^{-\tilde b_m})\sum_{n=1}^{N}G_{n}(y)e^{-n \tilde b_m}+G_{N}(y)e^{-(N+1) \tilde b_m}\\
&\leq &C\parallel g_{m}^{(m)}\parallel_{s_{m}, r_{m}}\sum_{n=1}^{\infty}n^{\tau}(e^{-n\tilde b_m}-e^{-(n+1)\tilde b_m})\\
&\leq & C\parallel g_{m}^{(m)}\parallel_{s_{m}, r_{m}}(\tilde b_m)^{-\tau}\\ & \leq &(400\ell)^\tau C \parallel g_{m}^{(m)}\parallel_{s_{m}, r_{m}} s_m^{-\tau}.
\end{eqnarray*}
It follows \begin{eqnarray}\label{6.?}
\parallel v\parallel_{s _{m}^{(1)}, r_{m}}&\leq &C\parallel g_{m}^{(m)}\parallel_{s_{m},r_{m}}s_{m}^{-\tau}\nonumber\\
&\leq & C\varepsilon_{m}s_{m}^{-\tau+d}=C\varepsilon_{m}s_{m}^{-\frac{\mu}{100}}\nonumber\\
&\ll& r_{m}.
\end{eqnarray}
Recall that we have a freedom to choose $\widehat{v}(0,0,y)$ such that
\be \label{li.3} -\widehat{v}(0,0,y)+\widehat{f}_{m}^{(m)}(0,0,y)=0,\;\;\forall y\in B_{\mathbb{C}}(r_{m}).\ee
In view of \eqref{eq*} and \eqref{li.3}, by applying the same method to \eqref{eq6.15}, we have
\begin{eqnarray} \label{li.4}
\parallel u\parallel_{s_{m}^{(2)}, r_{m}}&\leq &
 C\parallel v\parallel_{s_{m}^{(1)},r_{m}}s_{m}^{-\tau}+C\parallel f_{m}^{(m)}\parallel_{s_{m},r_{m}}s_{m}^{-\tau}\nonumber\\
 &\leq &C\varepsilon_{m}s_{m}^{-(\tau+\frac{\mu}{100})}.\end{eqnarray}
 By Cauchy estimate, we have that for $0\leq p,$ $0\leq q$ and $p+q=1,$
 \begin{eqnarray*}
 &&\parallel \partial_{x}^{p}\partial_{y}^{q}v\parallel_{s_{m}^{(2)}, r_{m}^{(1)}}\leq C\varepsilon_{m}s_{m}^{-\frac{\mu}{100}}\max\{s_{m}^{-1}, r_{m}^{-1}\}
 \leq C s_{m}^{l-\frac{\mu}{100}-(d+1+\frac{\mu}{100})}
  \leq Cs_{m}^{d-\frac{\mu}{100}}\ll 1,\\
 &&\parallel \partial_{x}^{p}\partial_{y}^{q} u\parallel_{s^{(3)}_{m},r^{(1)}_{m}}\leq C \varepsilon_{m}s_{m}^{-(\tau+\frac{\mu}{100})}\max\{s_{m}^{-1}, r_{m}^{-1}\}
 \leq Cs_{m}^{l-\tau-\frac{\mu}{100}-(d+1+\frac{\mu}{100})}\leq Cs_{m}^{\frac{\mu}{100}}\ll 1.
 \end{eqnarray*}
 By \eqref{eq6.2}, \eqref{6.?} and \eqref{li.4} and by means of implicit theorem, we have that $\Phi=\Psi^{-1}=\Phi_{m}:$
 \begin{eqnarray}\label{eq6.30}
 \Phi:\left\{
        \begin{array}{ll}
           x=\xi+U(\xi, \eta, t),\\
           y=\eta+V(\xi, \eta, t),
        \end{array}
      \right.
 \end{eqnarray}
where $(\xi, \eta, t)\in D(s_{m}^{(4)}, r_{m}^{(2)})$ and $U, V$ are real analytic in $D(s_{m}^{(4)}, r^{(2)}_{m}),$ and satisfy
\begin{eqnarray}
&& \parallel U\parallel_{s_{m}^{(4)},r_{m}^{(2)}}\leq C\varepsilon_{m}s_{m}^{-2\tau},\label{eq6.31}\\
&& \parallel V\parallel_{s_{m}^{(4)},r_{m}^{(2)}}\leq C\varepsilon_{m}s_{m}^{-\tau},\label{eq6.32}\\
&&\Psi(D(s_{m+1}, r_{m+1})) \subset\Psi(D(s_{m}^{(4)}, r_{m}^{(2)}))\subset D(s_{m}, r_{m})\label{eq6.33}
\end{eqnarray}
Now let us consider the reversibility of the changed system. First of all, by applying \eqref{eq6.17}, \eqref{eq6.18} to \eqref{eq6.15} and \eqref{eq6.16}, we have
\be\label{eq6.34}u(-x, y,-t)=-u(x, y, t),\;\;v(-x, y, -t)=v(x,y, t).\ee
Then by $\Phi \circ \Psi=id,$ we get
\begin{eqnarray}
&& u(x,y, t)+U(x+u(x, y, t), y+V(x, y, t), t)=0,\label{eq6.35}\\
&& v(x,y, t)+V(x+u(x, y, t), y+V(x, y, t), t)=0.\label{eq6.36}
\end{eqnarray}
By \eqref{eq6.34}, \eqref{eq6.35} and \eqref{eq6.36}, we have
\be\label{eq6.37}U(-\xi, \eta, -t)=-U(\xi, \eta, t),\;\;V(-\xi, \eta, -t)=V(\xi, \eta, t),\ee
where $(\xi, \eta, t)\in D(s_{m}^{(4)},r_{m}^{(2)}).$ It follows that the changed system are still reversible with respect to
$G: (x, y, t)\mapsto (-x, y, -t).$

\section{Estimates of new perturbations}
\begin{itemize}
  \item Estimate of \eqref{eq6.4}. By \eqref{6.?}, \eqref{li.3}, \eqref{li.4} and \eqref{eq*}, and regarding \eqref{eq6.4} as a function of $(x, y, t),$ we have
$$\parallel \eqref{eq6.4}(x,y,t)\parallel_{s_{m}^{(5)}, r_{m}^{(3)}}\leq \varepsilon_{m}^{\tilde{\mu}}\varepsilon_{m}=\varepsilon_{m+1}.$$
  \item Estimate of \eqref{eq6.6}. By Cauchy estimate, and in view of \eqref{li.4}, we have
\begin{eqnarray}
\parallel \eqref{eq6.6}(x, y, t)\parallel_{s_{m}^{(3)}, r_{m}^{(2)}}&\leq &\varepsilon_{m}s_{m}^{-(\tau+\frac{\mu}{100}+1)}r_{m}\\
&=& s_{m}^{\frac{4}{50}\mu}\varepsilon_{m}<\varepsilon_{m+1}.
\end{eqnarray}
\end{itemize}
Let
$$f_{m+1}^{0}=\eqref{eq6.4}+\eqref{eq6.6}.$$
Moreover, by \eqref{eq6.31} and \eqref{eq6.32}, we have
\be\label{eq6.50} \parallel f_{m+1}^{0}(\xi, \eta, t)\parallel_{s_{m+1}, r_{m+1}}\leq \varepsilon_{m+1}.\ee
Let $g_{m+1}^{0}=\eqref{eq6.9}+\eqref{eq6.11}.$ Similarly, we can prove
$$\parallel g_{m}^{0}(\xi, \eta, t)\parallel_{s_{m+1}, r_{m+1}}\leq \varepsilon_{m+1}s_{m+1}^{d}.$$
Thus we have proved the claim $(2)_{m}$ in the iterative Lemma with replacing $m$ by $m+1.$

Now we are in position to prove $(3)_{m}$ with replacing $m$ by $m+1.$ Let us return to the homological equation \eqref{eq6.15} and
\eqref{eq6.16}. Replacing $(x, y, t)$ by $(-x, y, -t),$ and in view of
$$f_{m}^{(m)}(-x, y, -t)=f_{m}^{(m)}(x,y ,t),\;\;g_{m}^{(m)}(-x, y, -t)=-g_{m}^{(m)}(x,y ,t).$$
We see that $u(-x, y, -t), -v(-x, y, -t)$ are still a pair of the solutions of \eqref{eq6.15} and \eqref{eq6.16}. Noting that the solutions of
\eqref{eq6.15} and \eqref{eq6.16} are unique.
So
$$u(-x, y, -t)=u(x, y, t), v(-x, y, -t)=-v(x, y, t).$$
This implies that \eqref{eq4.11} is reversible with respect to $G: (x, y, t)\mapsto (-x, y, -t).$
Thus $(3)_{m}$ holds true with replacing $m$ by $m+1.$

Again returning to \eqref{eq6.15} and \eqref{eq6.16}. Note that $f_{m}^{(m)}$ and $g_{m}^{(m)}$ are real for real arguments. It follows that
so are $u$ and $v$. Moreover, $f_{m+1}^{0}$ and $g_{m+1}^{0}$ are real for real arguments. We omit the detail here.
This completes the proof of the iterative Lemma.

\ \
\section{Application: Lagrange stability for a class of Li\'{e}nard equation}
Let $C>0$ be a universal constant which maybe different in different places. Consider the Li\'{e}nard equation
\begin{equation}\label{1}
\ddot{x}+x^{2n+1}+g(x, t)+f(x, t)\dot{x}=0,
\end{equation}
where $f$ and $g$ satisfy
\begin{itemize}
  \item [$(f)_{1}$] $f(x, t)$ is odd in $x,$ even in $t,$ and of period 1 in $t$,
  \item [$(f)_{2}$] $\exists \;0\leq p\leq n-1$, and $\exists$ integer $N>0$ such that
                    $$  |x^{k}\partial_{x}^{k}\partial_{t}^{\ell}f(x, t)|\leq C|x|^{p},\;\;|x|\gg 1,\;\;0\leq k\leq N,\;\;0\leq\ell\leq 2,$$
  \item [$(g)_{1}$] $g(x, t)$ is odd in $x,$ even in $t,$ and of period 1 in $t$,
  \item [$(g)_{2}$] $\exists\; q$ with $0\leq q\leq 2n-1$ such that
   $$|x^{k}\partial_{x}^{k}\partial_{t}^{\ell}g(x, t)|\leq C|x|^{q},\;\;|x|\gg 1, \;\;0\leq k\leq N,\;\;0\leq \ell \leq 2. $$
\end{itemize}

\ \
{Note that Eq. \eqref{1} is equivalent to the plane system
\begin{equation}\label{6}
\dot{x}=y,\;\;\dot{y}=-x^{2n+1}-g(x, t)-f(x, t)y.
\end{equation}
First of all, we consider a special Hamiltonian system
\begin{equation}\label{7}
\dot{x}=y, \;\;\dot{y}=-x^{2n+1}
\end{equation}
with Hamiltonian
\begin{equation}\label{8}
h(x,y)=\frac{y^{2}}{2}+\frac{x^{2(n+1)}}{2(n+1)}.
\end{equation}
Suppose that $(x_{0}(t), y_{0}(t))$ is the solution of Eq. \eqref{7}, with the initial conditions
$(x_{0}(0), y_{0}(0))=(0, 1).$ Clearly it is periodic. Let $T_{0}$ be its minimal positive period. It follows from \eqref{7} that
$x_{0}(t)$ and $y_{0}(t)$ possess the following properties:
\begin{itemize}
  \item [(a)] $x_{0}(t+T_{0})=x_{0}(t)$ and $y_{0}(t+T_{0})=y_{0}(t)$;
  \item [(b)] $x_{0}'(t)=y_{0}(t)$ and $y_{0}'(t)=-(x_{0}(t))^{2n+1}$;
  \item [(c)] $(n+1)((y_{0}(t))^{2}+(x_{0}(t))^{2n+2}=n+1$;
 \item [(d)]$x_{0}(-t)=-x_{0}(t)$ and $y_{0}(-t)=y_{0}(t).$
\end{itemize}
Following \cite{Dieckerhoff-Zehnder1987,Liu1998} we construct transformation $\psi: \mathbb{R}^{+}\times \mathbb{T} \rightarrow \mathbb{R}^{2}/\{0\},$
where $(x, y)=\psi(\lambda, \theta)$ with $\lambda>0$ and $\theta$ (mod 1) being given by the formula
\begin{equation}\label{9}
\psi: x=c^{\alpha}\rho^{\alpha}x_{0} (\frac{\theta T_{0}}{2\pi}),\;\;y=c^{\beta}\rho^{\beta}y_{0} (\frac{\theta T_{0}}{2\pi}),
\end{equation}
where $\alpha=\frac{1}{n+2},$ $\beta=1-\alpha$ and $c=\frac{2\pi}{\beta T_{0}}.$
By using the transformation $\psi,$ and noting properties (b) and (d), Eq. \eqref{6} is transformed into}
\begin{equation}\label{12}
\left\{
  \begin{array}{ll}
     \dot{\rho}=\frac{-1}{2\pi}(c\rho T_{0}y_{0}^{2}f(c^{\alpha}\rho^{\alpha}x_{0}, t)+c^{\alpha}\rho^{\alpha}y_{0}T_{0}g(c^{\alpha}\rho^{\alpha}x_{0}, t))\triangleq F_{1}( \theta, \rho, t),\\
     \dot{\theta}=c_{0}\rho^{2\beta-1}+c\alpha x_{0}y_{0}f(c^{\alpha}\rho^{\alpha}x_{0}, t)+c^{\alpha}{\alpha}\rho^{\alpha-1}x_{0}g(c^{\alpha}\rho^{\alpha}x_{0}, t)\triangleq c_{0}\rho^{2\beta-1}+F_{2}(\theta, \rho, t),
  \end{array}
\right.
\end{equation}
where $c_{0}=\beta c^{2\beta},$ $x_{0}=x_{0}(\frac{\theta T_{0}}{2\pi}),$ $y_{0}=y_{0}(\frac{\theta T_{0}}{2\pi}),$ ${2\beta-1=\frac{n}{n+2}}.$
\begin{dfn}
Given $\rho_{*}\gg 1.$ Consider $\rho\geq \rho_{*},$ $\theta\in \mathbb{T}.$ For  $\gamma\in\mathbb R,$ $q\geq 0,$ we call $y=y(\theta, \rho, t)\in P_{q,p}(\gamma)$
if
 $$\sup_{(\theta, \rho, t)\in \mathbb{T}\times [\rho_{*}, +\infty]\times
 \mathbb{T}}|\rho^{q-\gamma}\!\!\sum_{
  k+\ell\leq q, \;
k, \ell\geq 0
 }\partial_{\theta}^{k}\partial_{\rho}^{\ell}\partial_{t}^{p}y(\theta, \rho, t)|\leq C<\infty.$$
In light of $(f)_{1}$, $(f)_{2}$ and $(g)_{1}$, $(g)_{2},$ we have
\begin{eqnarray}
&& F_{1}\in P_{N, 2}(\frac{2n+1}{n+2}), F_{1}(-\theta, \rho, t)=-F_{1}(\theta, \rho, t), F_{1}(\theta, \rho, -t)=F_{1}(\theta, \rho, t),\\
&&\forall (\theta, \rho, t)\in \mathbb{T}\times [\rho_{*}, +\infty)\times \mathbb{T},\nonumber\\
&& F_{2}\in P_{N, 2}(\frac{n-1}{n+2}),F_{2}(-\theta, \rho, t)=F_{2}(\theta, \rho, t), F_{2}(\theta, \rho, -t)=F_{2}(\theta, \rho, t),\\
&&\forall (\theta, \rho, t)\in \mathbb{T}\times [\rho_{*}, +\infty)\times \mathbb{T}.\nonumber
\end{eqnarray}
For $\forall C>0,$ we define the domain
$$D_{C}=\{(\theta, \lambda, t)\mid \theta\in\mathbb{T}, t\in \mathbb{T}, \lambda\geq C\}.$$
\end{dfn}
\begin{lemma}\label{lem100}
There exists a diffeomorphism depending periodically on $t,$
$$\Psi: \rho=\mu+U(\phi, \mu, t),\;\;\theta=\phi+V(\phi, \mu, t)$$
such that
$$D_{c_{+}}\subset \Psi(D_{c_{0}})\subset D_{c_{-}}\;\;\mbox{for}\;\;1\ll c_{+}<c_{0}<c_{-},$$
and \eqref{12} is changed by $\Psi$ into
\begin{equation*}\label{*}
\left\{
  \begin{array}{ll}
     \dot{\mu}=\widetilde{F}_{1}(\phi, \mu, t),\\
     \dot{\phi}=c_{0}\mu^{2\beta-1}+h(\mu, t)+\widetilde{F}_2(\phi, \mu, t),
  \end{array}
\right.
\end{equation*}
and
\begin{eqnarray}
&& \widetilde{F}_{1}\in P_{5,0}\left(\frac{-1}{n+2}\right),\;\;\widetilde{F}_{1}(-\phi, \mu, -t)=-\widetilde{F}_{1}(\phi, \mu, t),\label{30}\\
&& \widetilde{F}_{2}\in P_{5,0}\left(\frac{-1}{n+2}\right),\;\;\widetilde{F}_{2}(-\phi, \mu, -t)={\widetilde{F}_{2}(\phi, \mu, t)},\label{31}\\
&& h\in P_{5,0}\left(\frac{n-1}{n+2}\right),\;\;\mbox{which is independant of}\;\;\phi\; \mbox{and}\;h(\mu, -t)=-h(\mu, t).
\end{eqnarray}
\end{lemma}
\begin{proof}
The proof is similar to that in Propositions 3.2 and 3.3 in \cite{Liu1998}.
\end{proof}
Let $\lambda=c_{0}\mu^{2\beta-1}.$ Then \eqref{*} reads
\begin{equation}\label{**}
\left\{
  \begin{array}{ll}
    \dot{\lambda}=F_{1}^{*}(\phi, \lambda, t),\\
     \dot{\phi}=\lambda+\widetilde{h}(\lambda, t)+F_{2}^{*}(\phi, \lambda, t),
  \end{array}
\right.
\end{equation}
where
\begin{eqnarray}
&& F_{1}^{*}\in P_{5,0}\left(\frac{-1}{n}\right),\;\;F_{1}^{*}(-\phi, \lambda, -t)=-F_{1}^{*}(\phi, \lambda, t),\label{35}\\
&& F_{2}^{*}\in P_{5,0}\left(\frac{-1}{n}\right),\;\;F_{2}^{*}(-\phi, \lambda, -t)=F_{2}^{*}(\phi, \lambda, t),\label{36}\\
&& \widetilde{h}(\lambda, t)\in P_{5,0}\left(\frac{n-1}{n}\right),\;\; h(\lambda, -t)=-h(\lambda, t).\label{37}
\end{eqnarray}
Following Lemma 4.1 in \cite{Liu1998}, we get the Poinc\'{a}re map of \eqref{**} is of the form
$$P: \;\theta_{1}=\theta+r(\lambda)+\widetilde{f}(\theta, \lambda),\;\lambda_{1}=\lambda+\widetilde{g}(\theta, \lambda),\;\lambda\geq \lambda_{*}\gg 1,$$
where
$r(\lambda)=\lambda+\int_{0}^{1}\widetilde{h}(\lambda, t)dt,$ $$\left|\sum_{k+\ell\leq 5}D_{\lambda}^{k}\partial_{\theta}^{\ell}\widetilde{f}\;\right|\leq C \lambda^{-\frac{1}{n}},\;\;\left|\sum_{k+\ell\leq 5}D_{\lambda}^{k}\partial_{\theta}^{\ell}\widetilde{g}\;\right|\leq C \lambda^{-\frac{1}{n}}.$$
Let $\rho=r(\lambda).$ By \eqref{37} and the implicit function theorem, we have that $P$ is of the form
$$P: \theta_{1}=\theta+\rho+f^{*}(\theta, \rho),\;\;\rho_{1}=\rho+g^{*}(\theta, \rho),$$
where
$$\left|\sum_{k+\ell\leq 5}D_{\rho}^{k}\partial_{\theta}^{\ell}f^{*}\;\right|\leq \rho^{-\frac{1}{n}}<\varepsilon_{0},\;\;
\left|\sum_{k+\ell\leq 5}D_{\rho}^{k}\partial_{\theta}^{\ell}g^{*}\;\right|\leq  \rho^{-\frac{1}{n}}<\varepsilon_{0},\;\;\rho\in [m, m+1],\;\;m\gg 1.$$
Let $\Phi= \psi\circ\Psi$. It is easy to see that equation \eqref{**} is reversible with respect to $G:\; (\phi, \lambda, t)\mapsto (-\phi, \lambda, -t)$. By Lemma 2.2 in \cite{Liu1998}, we have that $P$ is reversible with respect to $G$.

Using Theorem \ref{thm1}, we get that $P$ has an invariant carve $\Gamma_{m}\subset \mathbb{T}\times [m, m+1]$ $(\forall m\gg 1).$
It follows that the original equation has a family of invariant curves which are around the infinity. Thus, we have that
$$\sup_{t\in \mathbb{R}}|x(t)|+|\dot{x}(t)|\leq C,$$
where $(x, \dot{x})$ is the solution of \eqref{6}, and $C$ depends on initial $(x(0),\dot{x}(0)).$

\section*{Acknowledgements}
The authors are grateful to Professor B. Liu for his helpful discussions.
The work was supported in part by National
Natural Science Foundation of China (11601277 (Li), 11771253 (Qi), 11790272 (Yuan) and 11771093(Yuan)).



\begin{thebibliography}{00}

\bibitem{Arnold1963} V. I. Arnold, Proof of a theorem of an Kolmogorov on the invariance of quasi-periodic motions
under small perturbations of the hamiltonian. {\it Russian Mathematical Surveys} {\bf 18} (5), 9-36, 1963

\bibitem{Arnold1984}
V. I. Arnold, Reversible systems,
Nonlinear and turbulent processes in physics, Acad. Publ., New York, 1161-1174, 1984.

\bibitem{Arnold1986}
V. I. Arnold, M. B. Sevryuk, Oscillations and bifurcations in reversible systems,
 Nonlinear phenomena in plasma physics and hydrodynamics (RZ Sagdeev, Ed.), Physics Series, Mir, Moscow, 1986.


\bibitem{Broer2009}  H. W., Broer, G. B.,Huitema, and  M. B., Sevryuk, Quasi-periodic motions in families of dynamical systems: order amidst chaos. Springer, 2009.


\bibitem{Cheng} C. Q. Cheng, L. Wang, Destruction of Lagrangian torus for positive definite Hamiltonian systems,
                {\it Geometric and Functional Analysis}, {\bf 23} (3), 848-866, 2013,.

\bibitem{Chierchia-Qian} L. Chierchia, D. Qian, Moser's theorem for lower dimensional tori, {\it J. Diff. Eqs.}, {\bf 206} (1), 55-93, 2004.

\bibitem{Dieckerhoff-Zehnder1987}
R. Dieckerhoff, E. Zehnder,
Boundedness of solutions via twist theorem,
{\it Ann. Scula. Norm. Sup. Pisa},
{\bf 14} (1), 79-95, 1987.

\bibitem{Graef}
 J. R. Graef, On the generalized Lienard equation with negative damping, {\it J. Diff. Eqs. } {\bf 12}, 34-62, 1972.

\bibitem{Kol} A. N. Kolmogorov. On conservation of conditionally periodic motions for a small change in
Hamilton's function. {\it  Dokl. Akad. Nauk SSSR (N.S.)}  {\bf 98}:527-530, 1954.


\bibitem{Her83} M. R. Herman. Sur les courbes invariantes par les diff\'{e}omorphismes de l¡¯anneau,
                {\it Ast\'{e}risque}, {\bf 103-104}, 1-221, 1983.

\bibitem{Laederich-Levi1991}
S. Laederich, M. Levi,
Invariant curves and time-dependent potential,
{\it Ergod. Th. and Dynam. Sys.},
{\bf 11} (2), 365-378, 1991.

\bibitem{Levi1991} M. Levi, Quasi-periodic motions in superquadratic time-periodic potentials, {\it Comm. Math.
Phys.} {\bf 143(1)}, 43-83, 1991.


\bibitem{Levinson1943}
N. Levinson, On the existence of periodic solutions for second order differential
equations with a forcing term, {\it J. Math. Phys.}, {\bf 32}, 41-48, 1943.


\bibitem{Liu1989}
B. Liu, Boundedness for solutions of nonlinear Hill's equations with periodic forcing terms via Moser¡¯s
twist theorem, {\it J. Diff. Eqs.}, {\bf 79}, 304-315, 1989.

\bibitem{Liu1992}
B. Liu, Boundedness for solutions of nonlinear periodic differential equations via Moser¡¯s twist theorem,
{\it Acta Math Sinica (N.S.)}, {\bf 8}, 91-98, 1992.

\bibitem{Liu1991}
B. Liu, An application of KAM theorem of reversible systems,
{\it Sci. Sinica Ser. A.}, {\bf 34} (9), 1068-1078, 1991.



\bibitem{Liu1998}
B. Liu, F. Zanolin, Boundedness of solutions of nonlinear differential equations,
{\it J. Diff. Eqs.} {\bf 144} (1), 66-98, 1998.

\bibitem{Liu2005} B. Liu, Invariant curves of quasi-periodic reversible mappings, {\it Nonlinearity} {\bf 18},   685-701, 2005.


\bibitem{Moser1962} J.  Moser, On invariant curves of area-preserving mappings of an annulus,
 Nachr. Akad. Wiss. Gottingen Math. -Phys. KI. II,  1-20,1962.

\bibitem{Moser1969} J. Moser, On the construction of almost periodic solutions for ordinary differential equations, 1970
Proceedings of the International Conference on Functional Analysis and Related Topics, Tokyo, 1969,
University of Tokyo Press, Tokyo, 60-67.


\bibitem{Moser1973}
J. Moser, Stable and random motion in dynamical systems: with special emphasis on celestial mechanics,
Princeton Uni. Press, 1973.

\bibitem{Nash1956}
J. Nash, The imbedding problem for Riemannian manifolds, {\it  Ann. of Math. }  {\bf 63} (2), 20-63,  1956.


\bibitem{Piao2008}
D. Piao, W. Li, Boundedness of solutions for reversible system via Moser's twist theorem,
{\it J. Math. Anal. Appl.}, {\bf 341} (2), 1224-1235, 2008.

\bibitem{Poschel1982} J. P\"oschel, Integrability of Hamiltonian systems on Cantor sets,  {\it Comm. Pure Appl. Math.}  {\bf 35},
653-95, 1982

\bibitem{Reuter1951}
G. E. H. Reuter, A boundedness theorem for non-linear differential equations of
the second order, {\it Proc. Cambridge Phil. Soc.}, {\bf 47}, 49-54, 1951.

\bibitem{Russmann1970} H. R\"{u}ssmann, I. Kleine Nenner, \"Uber invariante Kurven differenzierbarer Abbildungen eines
Kreisringes, Nachr. Akad. Wiss. G\"otingen Math.-Phys. KI. I, 67-105, 1970.



\bibitem{Russmann1975} H. R\"{u}ssmann,  On optimal estimates for the solutions of linear partial differential
equations of first order with constant coefficients on the torus.
Dynamical systems, theory and applications. Springer, Berlin, Heidelberg, 598-624, 1975.

\bibitem{Russmann1983} H. R\"{u}ssmann,  On the existence of invariant curves of twist mappings of an annulus.
Geometric dynamics. Springer, Berlin, Heidelberg, 677-718, 1983.
\bibitem{Salamon2004}D. Salamon, The Kolmogorov-Arnold-Moser theorem. Math. Phys. Electron. J, {\bf 10} (3), 1-37, 2004.

\bibitem{Sevryuk1986}
M. B. Sevryuk, Reversible Systems, Lecture Notes in Math, 1211, Springer-Verlag, New York/Berlin,
1986.

\bibitem{Sevryuk95}
M. B. Sevryuk, The iteration-approximation decoupling in the reversible KAM theory, {\it Chaos}, {\bf 5} (3), 552-565, 1995.

\bibitem{Sevryuk98}
 M. B. Sevryuk, The finite-dimensional reversible KAM theory, {\it Physica D}, {\bf 112} (1-2), 132-147, 1998.

\bibitem{Sevryuk2011}
M. B. Sevryuk, The reversible context 2 in KAM theory: the first steps, {\it Regular and Chaotic Dynamics}, {\bf 16} (1-2), 24-38, 2011.

\bibitem{Sevryuk2012}
 M. B. Sevryuk, KAM theory for lower dimensional tori within the reversible context 2, {\it Moscow Mathematical Journal}, {\bf 12} (2), 435-455, 2012.




\bibitem{Rong2001}
R. Yuan, X. Yuan, Boundedness of solutions for a class of nonlinear differential equations
 of second order via Moser's twist theorem, {\it Nonlinear
Anal.} {\bf 46} (8), 1073-1087, 2001.


\bibitem{Yuan1995}
X. Yuan, Invariant tori of Duffing-type equations,
{\it Advances in Math. (China)},
{\bf 24} (4), 375-376, 1995.



\bibitem{Yuan1998}
X. Yuan,
Invariant tori of Duffing-type equations,
{\it J. Diff. Eqs.},
{\bf 142} (2), 231-262, 1998.

\bibitem{Yuan2000}
X. Yuan, Lagrange stability for Duffing-type equations, {\it J. Diff. Eqs.}, {\bf 160} (1), 94-117, 2000.

\bibitem{Zehnder1976} E. Zehnder,  Generalized implicit function theorems with applications to small divisor problems I and II, {\it Comm. Pure Appl. Math.} {\bf 28 }  91-140, 1975; {\bf 29}  49-113, 1976.

 \bibitem{Zhang2008} J., Zhang,  On lower dimensional invariant tori in $C^{d}$ reversible systems.{\it  Chin. Ann. Math.} Series B, {\bf 29} (5), 459-486, 2008.
\end{thebibliography}
\end{document}